\documentclass[pdflatex,sn-mathphys-ay]{sn-jnl}
\usepackage{lmodern} 

\usepackage{amssymb}
\usepackage{amsmath}
\usepackage{amsthm}
\usepackage{color}
\usepackage{url}
\usepackage{nicefrac}

\DeclareMathOperator{\prob}{P}
\DeclareMathOperator{\E}{E} 
\DeclareMathOperator{\Var}{Var}
\DeclareMathOperator{\Cov}{Cov}
\DeclareMathOperator{\sign}{sign}
\DeclareMathOperator{\Sd}{sd}

\DeclareMathOperator{\tr}{tr}

\newcommand{\dd}[1]{\mathrm{d}#1}

\usepackage{mleftright} 
\newcommand\pr[1]{\prob\mleft(\,#1\,\mright)}
\newcommand\ex[2][]{\E_{#1}\mleft[\,#2\,\mright]} 

\newcommand\var[1]{\Var\mleft[\,#1\,\mright]}
\newcommand\sd[1]{\Sd\mleft[\,#1\,\mright]}

\newcommand{\vertiii}[1]{{\left\vert\kern-0.15ex\left\vert\kern-0.15ex\left\vert #1 \right\vert\kern-0.15ex\right\vert\kern-0.15ex\right\vert}}

\newcommand{\bs}[1]{\boldsymbol{#1}}
\newcommand{\sigmamin}{\sigma_{\min}}

\usepackage{thmtools}
\usepackage{thm-restate}

\newtheoremstyle{spaced}
  {5pt}   
  {5pt}   
  {\itshape}  
  {}       
  {\bfseries} 
  {.}      
  {1em}    
  {}       

\declaretheorem[name=Theorem,style=spaced,numberwithin=section]{theorem}
\declaretheorem[name=Lemma,style=spaced,sibling=theorem]{lemma} 

\newcommand\bfX{\mathbf{X}}

\usepackage{booktabs}
\usepackage[separate-uncertainty=true]{siunitx} 

\usepackage{caption}
\usepackage{subcaption}

\usepackage{minitoc} 

\begin{document}

\title{On the distance between mean and geometric median in high dimensions}

\author*[1]{\fnm{Richard} \sur{Schwank}}\email{richard.schwank@tum.de}

\author[1,2]{\fnm{Mathias} \sur{Drton}}\email{mathias.drton@tum.de}

\affil*[1]{\emph{\orgdiv{School of Computation, Information and Technology}}, \orgname{Technical University of Munich},
\country{Germany}}

\affil[2]{\emph{\orgdiv{Munich Center for Machine Learning}}, \city{Munich}, \country{Germany}}

\abstract{
  The geometric median, a notion of ``center'' for multivariate distributions, has gained recent attention in robust statistics and machine learning. Although conceptually distinct from the mean (i.e., expectation), we demonstrate that both are very close in high dimensions when the dependence between the distribution components is suitably controlled. Concretely, we find an upper bound on the distance that vanishes with the dimension asymptotically, and derive a rate-matching first order expansion of the geometric median components. Simulations illustrate and confirm our results.
}


\maketitle

\doparttoc 
\faketableofcontents 
\part{} 
\vspace{-2.3cm} 

\section{Introduction}
The geometric median 
provides a robust alternative to the sample mean, mitigaging the impact of outliers. The improved robustness comes at the cost of a bias: in the infinite-data limit, the geometric median is generally distinct from the mean (i.e., expectation) of the data generating distribution, which often is the learning target in practice\footnote{For special distributions like the Gaussian, geometric median and mean can coincide. 
}. A method to balance robustness and bias is the median-of-means \citep{Lugosi2019, Minsker2015} or ``bucketing"  approach \citep{Guerraoui2024}.

In this paper, we take a step back and quantify the actual bias of the geometric median, focusing on high-dimensional distributions. We prove, when each distribution component only depends on a fixed number of neighboring components, that the maximum-norm distance between geometric median and mean vector is of order $\mathcal{O}(\tfrac{1}{p})$, where $p$ denotes the dimension. Furthermore, we derive an explicit expansion of the geometric median components when $p$ tends to infinity, in particular showing $\mathcal{O}(\tfrac{1}{p})$ is sharp in general.

We discuss connections to the literature and implications of our results at the end of the paper. Proofs are deferred to the appendix.

\section{Background and notation}
Let $\bfX:=(X_1,\ldots, X_p)\in\mathbb{R}^p$ be a $p$-variate random vector, and $\mathbf{x}_1,\ldots,\mathbf{x}_n$ be independent samples from $\bfX$. We denote the Euclidean norm by $\|\cdot\|$. Consider the two objective functions \begin{align*}
    D(\mathbf{m}):=\ex{\|\bfX - \mathbf{m}\|} && \text{and} && \hat{D}(\mathbf{m}):=\frac{1}{n}\sum_{i=1}^n \|\mathbf{x}_i - \mathbf{m}\|.
\end{align*} Any vector $\mathbf{m}\in\mathbb{R}^p$ minimizing $D$ or $\hat{D}$ is called \emph{population} or \emph{sample} geometric median of $\bfX$, respectively. A minimizer exists under mild conditions; in fact, by replacing $D(\mathbf{m})$ with the equivalent objective $\ex{\|\bfX - \mathbf{m}\| - \|\bfX\|}$, moment conditions on $\bfX$ can be avoided. When the minimizer is unique, we write $\mathbf{m}$ for the population and $\hat{\mathbf{m}}$ for the sample geometric median. By translation invariance of $D$, one can shift $\mathbf{m}$ to the origin without loss of generality, however we keep $\mathbf{m}$ as a variable to make the distance between median and mean recognizable in the formulas.

The geometric median $\hat{\mathbf{m}}$ has light tails even under minimal moment assumptions on $\bfX$ \citep{Minsker2015}.  This property makes it popular in heavy-tailed location testing where it can help outperform alternatives like Hotelling's $T^2$ \citep{Moettoenen1995,Cheng2023arxiv}, and more generally in robust statistics and machine learning \citep{Guerraoui2024}.

We write $\bs{\mu}:=\ex{\bfX}$ for the population mean and $\sd{X_i}$ and $\var{X_i}$ for the standard deviation and variance of component $X_i$.  For $\mathbf{v}\in\mathbb{R}^p$, we let $\|\mathbf{v}\|_\infty := \max(|v_1|,\ldots, |v_p|)$. We use the Landau symbols $\mathcal{O}(\cdot)$ and $o(\cdot)$.

\section{Theoretical results}

This section examines the distance between mean and geometric median and obtains the results highlighted in Table~\ref{tab:overview}. We begin with our distributional assumptions. 

Recall that a sequence of random variables $(X_i)_{i\in\mathbb{N}}$ is $M$-dependent if $(X_i)_{i=1}^k$ and $(X_i)_{i=k+M+1}^{\infty}$ are independent for all $k\in\mathbb{N}$.  $M$-dependence of finite-dimensional random vectors $\bfX=(X_i)_{1\leq i\leq p}$ is defined analogously under restriction to subvectors of $\bfX$.
\begin{restatable}{definition}{DefDistr}\label{DefDistr}
Let $M\ge0$ be an integer, and let $c,\sigmamin>0$, $q\ge2$, and $C\ge1$ be reals.  The distribution of an $M$-dependent  random vector $\bfX=(X_i)_{1\leq i\leq p}$ or an $M$-dependent sequence $(X_i)_{i\in\mathbb{N}}$ belongs to the class $\mathcal{D}(M, q, c, \sigmamin, C)$ if all components $X_i$ admit a Lebesgue density bounded by $c$, with $\ex{|X_i|^q}\leq C$ and $\sd{X_i}\geq \sigmamin>0$.
\end{restatable}
Note that the definition implies that  $|\mu_i|\leq C $ and $ \var{X_i}\leq C$ as $C\ge 1$.

\begin{table}
    \centering
    \begin{tabular}{ccc}
    Result & Integrability & Reference\\\hline
    $\|\mathbf{m} - \bs{\mu}\|_\infty \leq \text{const}$ & $q>2$ & Lemma \ref{lemmaGmedComponentsBounded} \\
    $\|\mathbf{m} - \bs{\mu}\|_\infty = \mathcal{O}(\tfrac{1}{p})$ & $q\geq 3$ & Theorem \ref{theoremAllComponentsBound} \\
    $m_{p,i} = \mu_i - \tfrac{\alpha_i}{p} + o(\tfrac{1}{p})$ & $q > 3$ & Theorem \ref{theoremComponentExpansion}
    \\\hline
    \end{tabular}
    \caption{Theoretical results with required integrability $q$.}
    \label{tab:overview}
\end{table}

Let the $p$-variate random vector $\bfX$ belong to the class $\mathcal{D}\equiv\mathcal{D}(M, q, c, \sigmamin, C)$. Then, $\bfX$ is guaranteed to have a unique population geometric median $\mathbf{m}$ when $p\geq M+2$ (Lemma \ref{lemma_exists_unique_gmed}). We build on the following identity for $\mathbf{m}$, which is obtained by setting derivatives in the geometric median objective to zero:
\begin{equation}\label{eq:basic_equation}
    \mathbf{m} - \bs{\mu} = \ex{\frac{1}{\|\bfX - \mathbf{m}\|}}^{-1}\ex{\frac{\bfX - \bs{\mu}}{\|\bfX - \mathbf{m}\|}}.
\end{equation} Equation \eqref{eq:basic_equation} can be proven rigorously when $p$ is large enough (Lemma \ref{lemma_basic_inequality}).

The guiding intuition for why $\|\mathbf{m} - \bs{\mu}\|$ should be small is that, by $M$-dependence, $X_i-\mu_i$ and $\|\bfX - \mathbf{m}\|$ are nearly independent for large $p$. Were they actually independent, the second expectation in \eqref{eq:basic_equation} would factorize and imply $\mathbf{m} - \bs{\mu} =\mathbf{0}$, since $\ex{X_i - \mu_i}=0$.

Formalizing this approach requires controlling $\ex{\|\bfX - \mathbf{m}\|^{-1}}$, or slight variations introduced later, uniformly in the dimension $p$. To deal with $\mathbf{m}$ on the right hand side of \eqref{eq:basic_equation}, we use the standard a-priori bound $\|\mathbf{m} - \bs{\mu}\|^2 \leq C\,p$ (see Lemma \ref{lemmaInitialGmedMeanBound} or \citealt{Minsker2024GeometricMedian}).

\begin{restatable}{lemma}{lemmaUnifIntegr}\label{lemmaUnifIntegr}
    Let $\mathbf{Y}:= (Y_1,\ldots, Y_p)\in\mathcal{D}(M, q, c, \sigmamin, C)$ with $q>2$, and let $\mathbf{v}\in\mathbb{R}^p$ with $\|\mathbf{v}\|^2 \leq C_2\, p$ for some $C_2<\infty$. When $p\geq \lceil4k\rceil(M+1) + 1$,
    \begin{equation}
         \ex{\left(\frac{1}{\|\mathbf{Y} - \mathbf{v}\|^2}\right)^{k}}\leq \frac{\beta(k, C_2, \mathcal{D})}{\ex{\|\mathbf{Y} - \mathbf{v}\|^2}^k} < \infty
    \end{equation} for any $k>0$ and a constant $\beta$ depending only on $k, C_2$ and the $\mathcal{D}$-parameters.
\end{restatable}
Applying the Cauchy-Schwarz inequality to the second expectation in \eqref{eq:basic_equation} and inserting Lemma \ref{lemmaUnifIntegr} yields the following very useful intermediate bound. 

\begin{restatable}{lemma}{lemmaGmedComponentsBounded}\label{lemmaGmedComponentsBounded}
Let the $p$-variate $\bfX$ belong to $\mathcal{D}(M, q, c, \sigmamin, C)$ with $q>2$. Then, $\|\mathbf{m} -\bs{\mu}\|_\infty \leq B_1$ for a constant $B_1$ depending only on the $\mathcal{D}$-parameters.   
\end{restatable}

For deriving the main results, we measure the distance between each right hand side component of equation \eqref{eq:basic_equation} and zero in a special way. Precisely, let $i\leq p$ be fixed and define $\check{\bfX}:= (X_j)_{\{1\leq j\leq p\, :\, |j-i|>M\}}$. With $\check{\mathbf{m}}$ defined similarly, $\ex{\nicefrac{(X_i - \mu_i)}{\|\check{\bfX} - \check{\mathbf{m}}\|}}=0$ by $M$-dependence. Subtracting this zero from equation \eqref{eq:basic_equation}, we obtain after a bit of algebra that \begin{equation}\label{eq:basic_equation_improved}
    m_{i} - \mu_i = \ex{\frac{1}{\|\bfX - \mathbf{m}\|}}^{-1} \ex{\frac{-(X_i - \mu_i)\sum_{|j - i|\leq M}(X_j - m_{j})^2}{\|\bfX - \mathbf{m}\|\|\check{\bfX} - \check{\mathbf{m}}\|(\|\bfX - \mathbf{m}\| + \|\check{\bfX} - \check{\mathbf{m}}\|)}},
\end{equation}
as proven for $p$ large enough in Lemma \ref{lemmaBasicEqImproved}. The first expectation is $\mathcal{O}(\sqrt{p})$ while the second is $\mathcal{O}(p^{-3/2})$, leading to the first main result:

\begin{restatable}{theorem}{theoremAllComponentsBound}\label{theoremAllComponentsBound}
    Let the $p$-variate $\bfX$ belong to $\mathcal{D}(M, q, c, \sigmamin, C)$ with $q\geq 3$. Then, $\|\mathbf{m}- \bs{\mu}\|_\infty \leq \tfrac{B_3}{p}$ for $B_3$ depending only on the $\mathcal{D}$-parameters.
\end{restatable}

Next, we consider random vectors $\bfX_p\in\mathbb{R}^p$ and their medians $\mathbf{m}_p$ as $p$ tends to infinity. We establish an asymptotic expansion $m_{p,i} = \mu_i - \tfrac{\alpha_i}{p} + o(\tfrac{1}{p})$ under mild assumptions. This shows that the rate obtained in Theorem \ref{theoremAllComponentsBound} is generally sharp, and through explicit formulas for $\alpha_i$ we gain insight into which distributional properties of $\bfX$ influence $m_{p,i} - \mu_i$ most strongly.

We prove that $p\cdot(m_{p,i} - \mu_i)$ converges to a constant $\alpha_i$ as $p\to\infty$. Precisely, the random variables inside the expectations on the right hand side of the scaled equation \eqref{eq:basic_equation_improved} admit a limit in probability, since $m_{p,i}\to \mu_i$ by Theorem \ref{theoremAllComponentsBound}. We transfer this convergence in probability to convergence of the expectations via a uniform integrability argument, requiring $q>3$. 

\begin{restatable}{theorem}{theoremComponentExpansion}\label{theoremComponentExpansion}
    Let $(X_i)_{i\in\mathbb{N}}\in \mathcal{D}(M, q, c, \sigmamin, C)$ with $q > 3$ and assume $\tfrac{1}{p}\sum_{i=1}^p \var{X_i} \to \bar{\sigma}^2$. Fix $i\in\mathbb{N}$. Then, $m_{p,i} = \mu_i -\tfrac{\alpha_i}{p} + o(\tfrac{1}{p})$, where \begin{equation*}
        \alpha_i:=\frac{1}{2\bar{\sigma}^2} \sum_{|j-i|\leq M}\ex{(X_i - \mu_i)(X_j - \mu_j)^2}.
    \end{equation*}
\end{restatable}
The expansion coeffients $\alpha_i$ can be computed analytically in many cases. When $M=0$, i.e., $(X_i)_i$ are independent, the sum reduces to the skewness of $X_i$. Note that while $\sigmamin^2 \leq \tfrac{1}{p}\sum_{i=1}^p \var{X_i} \leq C$ always has a convergent subsequence, there may be multiple accumulation points.  
In this case, the expansion is satisfied on each convergent subsequence with individual expansion coefficients $\alpha_i$.

\section{Simulations}
We simulate $\hat{\mathbf{m}}_p$ for a range of distributions to validate Theorems \ref{theoremAllComponentsBound} and \ref{theoremComponentExpansion}. The simulation details are discussed at the end of this section.

\begin{figure}
    \centering
    \includegraphics[width=\linewidth]{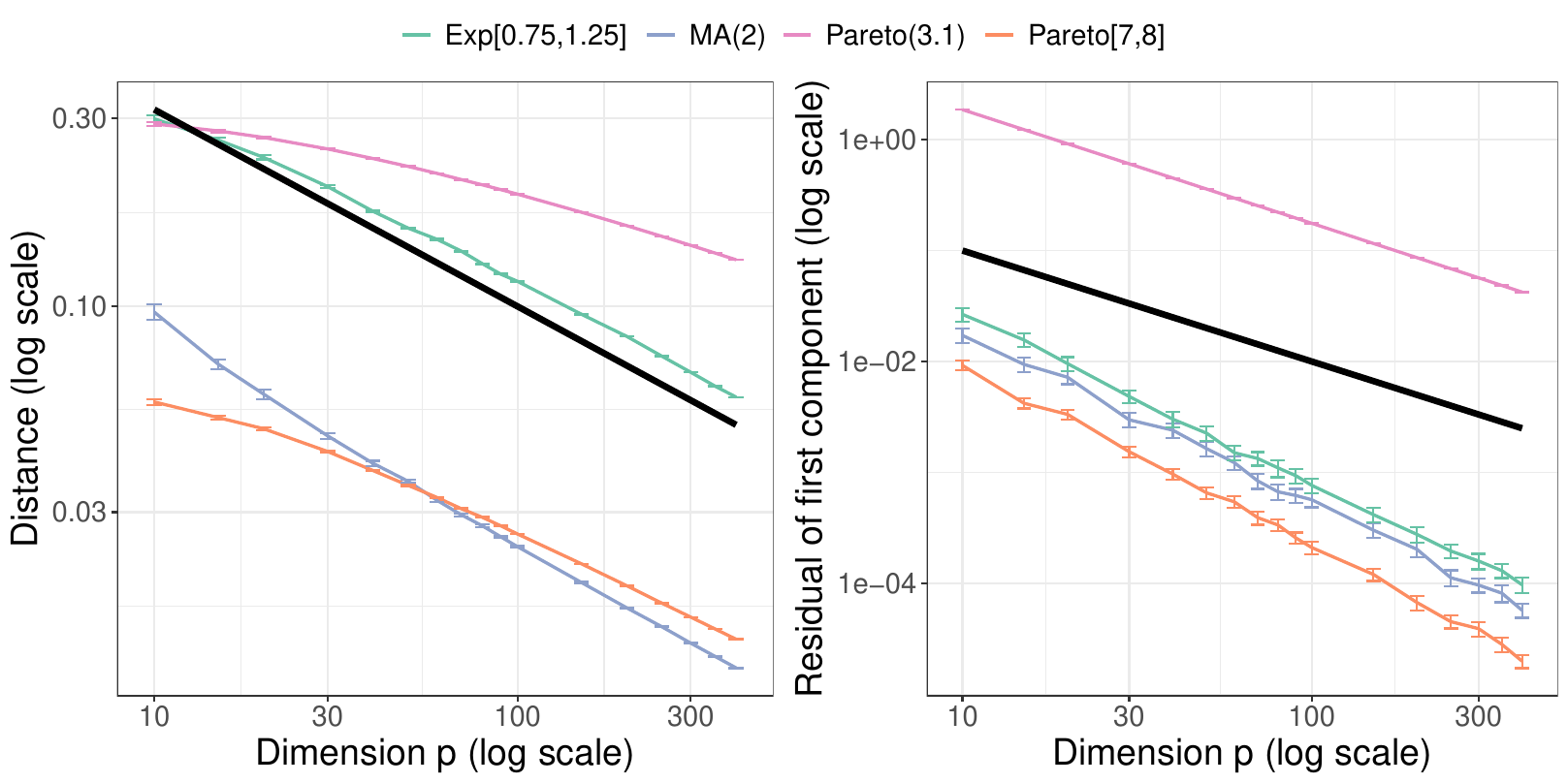}
    \caption{(Left) Euclidean distance $\|\hat{\mathbf{m}}_p - \bs{\mu}\|$. (Right) Residuals from Theorem \ref{theoremComponentExpansion} for $i=1$; see equation \eqref{eq:sim_residuals}. Black reference lines are $\mathcal{O}(\tfrac{1}{\sqrt{p}})$ (left) and $\mathcal{O}(\tfrac{1}{p})$ (right).}
    \label{fig:Simulation}
\end{figure}

We plot $\|\hat{\mathbf{m}}_p - \bs{\mu}\|$, for which Theorem \ref{theoremAllComponentsBound} predicts $\mathcal{O}(\tfrac{1}{\sqrt{p}})$, on the left of Figure \ref{fig:Simulation}. On the right, we plot the residual \begin{equation}\label{eq:sim_residuals}
    \big|\hat{m}_{p,1} - \mu_1 + \frac{1}{2\bar{\sigma}^2 p}\sum_{j=1}^M\ex{(X_1 - \mu_1)(X_j - \mu_j)^2}\big|,
\end{equation} for which Theorem \ref{theoremComponentExpansion} predicts $o(\tfrac{1}{p})$. We consider four data-generating distributions $(X_i)_{i\in\mathbb{N}}$: first, independent exponential random variables with rates uniformly sampled from $[0.75, 1.25]$, second an MA(2) process $X_{i} = L_i + 0.2 L_{i-1} - 0.1 L_{i-2}$ with iid skew normal noise $(L_i)_i$, third iid Pareto variables with shape parameter $3.1$ (meaning $\ex{|X_1|^{k}}<\infty$ iff $k<3.1$), and finally independent Paretos with shapes uniformly sampled from $[7,8]$. Table \ref{tab:slopes} lists linear regression slopes to the simulation curves after $p\geq 100$.

The simulation study clearly supports the theoretical predictions. For the heavy tailed Pareto distributions, $\|\hat{\mathbf{m}}_p - \bs{\mu}\|$ approaches $\mathcal{O}(\tfrac{1}{\sqrt{p}})$ rather slowly in the left hand plot. To see why, note $m_{p,1} - \mu_1 \approx \tfrac{-19.5}{p} + \tfrac{20.2}{p^{1.03}}$ for Pareto(3.1) from the component expansion plot on the right. Since all geometric median components are equal in the iid setting, $\|\mathbf{m}_p - \bs{\mu}\|\approx\sqrt{p}\cdot|\tfrac{-19.5}{p} + \tfrac{20.2}{p^{1.03}}|$, which indeed fits the simulations results (Figure \ref{fig:Pareto3Explanation}). The near cancellation considerably slows the decay for moderate dimension $p$.

To ensure that $\hat{\mathbf{m}}_p$ is sufficiently close to $\mathbf{m}_p$, we compute $\hat{\mathbf{m}}_p$ from $n := p^3$ samples. Then, for a large class of distributions and $s>0$, we have $\|\hat{\mathbf{m}}_p - \mathbf{m}_p\| \lesssim \sqrt{\tfrac{p}{n}} = \tfrac{1}{p}$ with probability $1-e^{4s}$ up to constants not depending on $p,n$ \citep[Thm. 3.9]{Minsker2024GeometricMedian}. While $n := p^3$ controls the simulation error, it introduces a memory bottleneck. Therefore, we compute $\hat{\mathbf{m}}_p$ by the online averaged Robbins-Monro algorithm with the default parameter choices $c_\gamma=2$ and $\alpha=\tfrac{3}{4}$ \citep{Cardot2017OnlineGeometricMedianEstimation}.

\section{Discussion}
The bias of the geometric median was recently highlighted by \citet{Minsker2024GeometricMedian}, where a constant upper bound on the Euclidean bias for a large class of distributions is derived. We complement this result by showing the Euclidean bias to be $\mathcal{O}(\tfrac{1}{\sqrt{p}})$ for a smaller class.  Our expansion in Theorem \ref{theoremComponentExpansion} always contains the term $-\ex{(X_i - \mu_i)^3}$, i.e., the negative component skewness. This aligns with statistical intuition for the \emph{univariate} median \citep{Hippel2005}.
Our result also links to the classic paper of \cite{Brown1983} showing that sample geometric median and mean are equivalent in terms of asymptotic variance for isotropic Gaussians as the ambient dimension approaches infinity.

Theorem \ref{theoremComponentExpansion} can also be applied to the geometric median of means, i.e., when $X_i = \tfrac{1}{K}\sum_{k=1}^K X_i^{(k)}$ with the sequences $(X_i)_{i\in\mathbb{N}}^{(k)}$ for $k=2,\ldots, K$ being iid copies of $(X_i)_{i\in\mathbb{N}}^{(1)}$. In this case, $\alpha_i= \nicefrac{\alpha^{(1)}_i}{K}$. In other words, the first order bias of the population geometric median of means scales \emph{exactly} as $\tfrac{1}{K}$ with the block size $K$. 
Understanding the population bias is one step towards characterizing the bias of the sample geometric median of means. 

The geometric median exhibits a benign bias for high-dimensional distributions beyond those considered in this paper, as empirically observed in \citet{schwank2025robustscorematching}. Towards extending the approach of this paper, we note that if through appropriate mixing conditions on $(X_i)_{i\in\mathbb{N}}$ one can show the convergence in probability $\tfrac{\|\bfX - \mathbf{m}_p\|^2}{\ex{\|\bfX - \mathbf{m}_p\|^2}} \to 1$ similar to Lemma \ref{lemma_conv_prob}, the right hand side components of equation \eqref{eq:basic_equation} converge to $1\cdot\ex{X_i - \mu_i} = 0$, provided one can exchange limit and expectation under a uniform integrability argument similar to the proof of Theorem \ref{theoremComponentExpansion}. This would indeed imply that $m_{p,i}$ converges to $\mu_i$ in more general settings.

Finally, we would like to emphasize that our theory supports existing literature by highlighting that robustness of the geometric median may hinge on the dependence structure of contamination patterns. 
Consider $X_i = (1-d_i)\cdot X_{0,i} + d_{i}\cdot Z_i$, where $\mathbf{d}\in \{0,1\}^p$ is a random vector independent of the rest. The mean is said to not be robust against this contamination, since $\mu_i = \pr{d_i=0}\mu_{0,i} + \pr{d_i=1}\ex{Z_i}$ depends strongly on $\ex{Z_i}$. When $\mathbf{d}$ is $M$-dependent, our theory applies under appropriate conditions on $\bfX_0$ and $\mathbf{Z}$, suggesting the geometric median need not be robust in such a setting. This is consistent with the observations of \citet{RAYMAEKERS2024Cellwise}, as $M$-dependence of $\mathbf{d}$ is a form of \emph{cellwise} contamination. On the other hand, when $\mathbf{d}$ is fully dependent, i.e., all $d_i$ are almost surely equal, also called \emph{rowwise} contamination, the geometric median is provably robust \citep{Minsker2015}. Indeed, our theory does \emph{not} apply in this case due to lacking $M$-dependence.

\vspace{0.7cm}\noindent\textbf{Acknowledgements} This paper is supported by the DAAD programme Konrad Zuse Schools of Excellence in Artificial Intelligence, sponsored by the Federal Ministry of Education and Research. 
Further, this work has been funded by the German Federal Ministry of Education and Research and the Bavarian State Ministry for Science and the Arts. The authors of this work take full responsibility for its content.

\newpage
\appendix
\renewcommand{\thesection}{\Alph{section}} 
\addcontentsline{toc}{section}{Appendix} 
\part{Supplementary Material} 
\parttoc 
\section{Proofs and additional Lemmas}
\subsection{Proof of Theorem \ref{theoremAllComponentsBound}}
\theoremAllComponentsBound*
\begin{proof}
    First, consider the case $p\geq 10M+8$. Let $i\in\{1,\ldots,p\}$. Note that \begin{equation}\label{eq:proof_allCompBound1}
        \E \big[|(X_i - \mu_i)\sum_{|i-j|\leq M}(X_j - m_{j})^2|\big] \leq B
    \end{equation} for a constant $B$ only depending on the $\mathcal{D}$-parameters (see Lemma \ref{lemmaNumeratorBound}). 
    
    We consider indices with sufficient distance to $i$, namely $\mathcal{I}:= \{j\in\{1,\ldots,p\} : |j - i|> 2M\}$. For $\mathbf{v}\in\mathbb{R}^p$, let $\mathbf{v}^* := (v_j)_{j\in\mathcal{I}}\in\mathbb{R}^{d}$, where $d := |\mathcal{I}|$. Lemma \ref{lemmaBasicEqImproved} implies that
    \begin{equation}\label{eq:proof_allCompBound2}
        |m_{i} - \mu_i| \leq B\cdot \ex{\frac{1}{\|\bfX - \mathbf{m}\|}}^{-1} \ex{\frac{1}{2\|\bfX^* - \mathbf{m}^*\|^{3}}},
    \end{equation} where we bounded $\|\bfX - \mathbf{m}\| \geq \|\check{\bfX} - \check{\mathbf{m}}\| \geq \|\bfX^* - \mathbf{m}^*\|$ and used that $\bfX^*$ is independent of $(X_j)_{|i-j|\leq M}$.

    Applying Jensen's inequality with the convex function $1/\sqrt{x}$, we have that \begin{equation}\label{eq:proof_allCompBound3}
        \ex{\frac{1}{\|\bfX - \mathbf{m}\|}}^{-1} \leq \sqrt{\ex{\|\bfX-\mathbf{m}\|^2}}.
    \end{equation}
    Next, apply Lemma \ref{lemmaUnifIntegr} to $\mathbf{Y}:=\bfX^*$ and $\mathbf{v}:=\mathbf{m}^*$ with $k:=\tfrac{3}{2}$. This is possible since $\|\mathbf{m}^*\|^2\leq \sum_{j\in\mathcal{I}} (|\mu_j| + |\mu_j - m_j|)^2 \leq (C + B_1)^2\cdot d$ by Lemma \ref{lemmaGmedComponentsBounded}, and since $d\geq p - (4M+1) \geq 6(M+1)+1$ by assumption on $p$. The resulting bound is: \begin{equation}
        \ex{\frac{\ex{\|\bfX^* -\mathbf{m}^*\|^2}^{3/2}}{\|\bfX^* - \mathbf{m}^*\|^{3}}} \leq \beta(\tfrac{3}{2}, (B_1 + C)^2, \mathcal{D}) =: \tilde{B}.
    \end{equation}
    Combining the bounds and using $\|\bs{\mu} - \mathbf{m}\|^2\leq p\,C$ by Lemma \ref{lemmaInitialGmedMeanBound}, we find:
    \begin{multline*}
        |m_{i} - \mu_i| \leq \frac{B \tilde{B}}{2} \cdot \frac{\sqrt{\ex{\|\bfX - \mathbf{m}\|^2}}}{\ex{\|\bfX^* - \mathbf{m}^*\|^2}^{3/2}} \leq \frac{B\tilde{B}}{2} \frac{\sqrt{\tr{\var{\bfX}} + \|\bs{\mu} - \mathbf{m}\|^2}}{(\tr{\var{\bfX^*}} + 0)^{3/2}} \\
        \leq \frac{B\tilde{B}}{2} \frac{\sqrt{pC + pC}}{((p - (4M+1))\sigmamin^2)^{3/2}} \leq \frac{1}{p} \cdot\frac{B\tilde{B}\sqrt{2C}}{2\sigmamin^3 0.5^{3/2}},
    \end{multline*} where we used $4M+1 \leq \tfrac{p}{2}$ by the assumptions on $p$. As $i$ was arbitrary, the theorem claim is proven for $p\geq 10M+8$.

    When $p\leq 10M+8$, the claim holds with $B_3 := (10M+8)\cdot B_1$ by Lemma \ref{lemmaGmedComponentsBounded}. The maximum of both derived constants is a valid choice for $B_3$.
\end{proof}
\subsection{Proof of Theorem \ref{theoremComponentExpansion}}
\theoremComponentExpansion*
\begin{proof}
    We show the following identity, which is equivalent to the claim: \begin{equation*}
        p\cdot(m_{p,i} - \mu_i) \xrightarrow{p\to\infty} \frac{-1}{2\bar{\sigma}^2}\sum_{|j-i|\leq M} \ex{(X_i - \mu_i)(X_j - \mu_j)^2}.
    \end{equation*}
    For $p$ large enough, Lemma \ref{lemmaBasicEqImproved} implies that\begin{equation*}
        p\cdot(m_{p,i} - \mu_i) = \frac{p \sqrt{\ex{\|\bfX - \mathbf{m}_p\|^2}}}{\ex{\|\bfX - \mathbf{m}_p\|^2}^{3/2}}\, 
        \cdot\, \ex{F_p^{-1}}^{-1} \ex{\frac{-(X_i-\mu_i)\cdot S_p}{F_p\cdot \check{F}_p \cdot (F_p + \check{F}_p)}} =: I_p\, \cdot \,II_p,
    \end{equation*} where we defined the two fractions \begin{align*}
        F_p &:= \sqrt{\frac{\|\bfX - \mathbf{m}_p\|^2}{\ex{\|\bfX - \mathbf{m}_p\|^2}}}, & \check{F}_p &:= \sqrt{\frac{\|\check{\bfX} - \check{\mathbf{m}}_p\|^2}{\ex{\|\bfX - \mathbf{m}_p\|^2}}},
    \end{align*} and the sum $S_p := \sum_{|j-i|\leq M}(X_j - m_{p,j})^2$. 

    For the first factor $I_p$, note $\ex{\|\bfX - \mathbf{m}_p\|^2} = \tr(\var{\bfX}) + \|\bs{\mu} - \mathbf{m}_p\|^2$. Further, $\|\bs{\mu} - \mathbf{m}_p\|^2\leq B_3^2/p$ by Theorem \ref{theoremAllComponentsBound}, implying that \begin{equation*}
        I_p^{-1} =\frac{1}{p}\ex{\|\bfX - \mathbf{m}_p\|^2} = \frac{1}{p}\sum_{i=1}^p \var{X_i} + \mathcal{O}\left(\frac{1}{p^2}\right) \xrightarrow{p\to\infty}\bar{\sigma}^2.
    \end{equation*} Consequently, $\lim_{p\to\infty} I_p = \tfrac{\bar{\sigma}}{\bar{\sigma}^{3}} = \tfrac{1}{\bar{\sigma}^2}$.

    To determine the limit of the second factor $II_p$, we adopt a two-step approach. In the first step, we determine the limit in probability of the expectation arguments. Using Theorem \ref{theoremAllComponentsBound}, the deterministic sequence $(m_{p,j})_{p\in\mathbb{N}}$ converges to $\mu_j$ in particular for all $j$ with $|i-j|\leq M$. Thus: \begin{equation*}
        S_p \to \sum_{|j - i|\leq M} (X_j - \mu_j)^2 =: S
    \end{equation*} in probability as $p\to\infty$. 
    Further, $F_p^2 \to 1$ in probability as shown in Lemma \ref{lemma_conv_prob}. We conclude $F_p = |F_p|\to 1$ by the continuous mapping theorem. Finally, \begin{equation*}
        \check{F}_p^2 = F_p^2  - S_p \cdot\frac{1}{\ex{\|\bfX - \mathbf{m}_p\|^2}} \to 1 - S\cdot0 = 1
    \end{equation*} by $\ex{\|\bfX - \mathbf{m}_p\|^2}^{-1}\leq (p\, \sigmamin^2)^{-1}\to 0$ and Slutsky's theorem. Using the vector continuous mapping theorem and that convergence in probability of a random vector is equivalent to the component convergence, we conclude that \begin{align}\label{eq:proof_expansion1}
        \frac{- (X_i - \mu_i)\,S_p}{F_p\cdot \check{F}_p \cdot (F_p + \check{F}_p)} &\to \frac{-(X_i - \mu_i)\, S}{2} & &\text{and}& F_p^{-1}&\to 1.
    \end{align}

    In the second step, we strengthen the convergence in probability to convergence of the expectations via an uniform integrability argument \citep[Thm 6.2]{DasGupta2008}. First, $\sup_{p\geq 4(M+1)+1} \ex{(F_p^{-1})^2}<\infty$ by Lemma \ref{lemmaUnifIntegr}, implying $\ex{F_p^{-1}}\to \ex{1} = 1$ as $p\to\infty$. Further, for $p$ large enough: \begin{multline*}
        \ex{\left|\frac{ (X_i - \mu_i)\,S_p}{F_p\cdot \check{F}_p \cdot (F_p + \check{F}_p)}\right|^{q/3}} \leq B_4(q,\mathcal{D}) \cdot  \ex{\frac{\ex{\|\bfX - \mathbf{m}_p\|^2}^{q/2}}{(2 \|\bfX^* - \mathbf{m}_p^*\|^3)^{q/3}}} \\
        \leq \frac{B_4\cdot \beta}{2^{q/3}} \frac{\ex{\|\bfX - \mathbf{m}_p\|^2}^{q/2}}{\ex{\|\bfX^* - \mathbf{m}_p^*\|^2}^{q/2}}
        \leq \frac{B_4\cdot \beta}{2} \left(\frac{2p\,C}{(p - (4M+1))\sigmamin^2}\right)^{q/2},
    \end{multline*} where similar to the proof of Theorem \ref{theoremAllComponentsBound} we used independence of $(X_i - \mu_i)S_p$ and $\bfX^*$, then applied Lemma \ref{lemmaNumeratorBound} (introducing $B_4$), then Lemma \ref{lemmaUnifIntegr} applied to $\bfX^*$ (introducing $\beta$), and finally the trace-bound similar to Theorem \ref{theoremAllComponentsBound}. The final expression remains bounded as $p\to\infty$, proving uniform integrability with exponent $q/3 > 1$. We may therefore exchange limit and expectation in both terms of $II_p$, and the theorem's claim follows from \eqref{eq:proof_expansion1}.
\end{proof}

\subsection{Proof of Lemma \ref{lemmaUnifIntegr}}
\lemmaUnifIntegr*
\begin{proof} Let $k>0$ and $p\geq \lceil4k\rceil(M+1) + 1$. To simplify notation and stress the additive structure, we denote $S(\mathbf{v}):= \sum_{i=1}^p (Y_i - v_{i})^2 = \|\mathbf{Y} - \mathbf{v}\|^2$ and $s:= \ex{S(\mathbf{v})}$. We repeatedly make use of the following bounds on $s$: \begin{equation}\label{eq:proof_unifIntegr0}
    \sigmamin^2 p\leq s \leq \sum_{i=1}^p 2(\ex{Y_i^2} + v_i^2) \leq 2 (C+C_2)\,p,
\end{equation} where we used $\ex{(Y_i - v_{i})^2}\geq \ex{(Y_i - \ex{Y_i})^2} = \sigmamin^2$ and Cauchy's inequality $|ab|\leq \tfrac{a^2}{2} + \tfrac{b^2}{2}$. In particular, it holds that $s>0$ and we therefore may prove the Lemma's claim by bounding $\ex{(\tfrac{s}{S(\mathbf{v})})^k}$. Using $\ex{Z}=\int_0^\infty \pr{Z\geq t}\dd{t}$ for any nonnegative random variable $Z$ (even if the expectation should be infinite), we obtain that \begin{multline}\label{eq:proof_unifIntegr_decomp}
     \ex{\left(\frac{s}{S(\mathbf{v})}\right)^k} = \int_0^{r^k} \pr{S(\mathbf{v}) \leq s\cdot t^{-1/k}}\dd{t}  + \int_{r^k}^{R^k} \pr{S(\mathbf{v}) \leq s\cdot t^{-1/k}}\dd{t}  \\ + \int_{R^k}^\infty \pr{S(\mathbf{v}) \leq s\cdot t^{-1/k}}\dd{t}=: I_1 + I_2 + I_3,
 \end{multline} where we split the positive reals at $r^k\leq R^k$ defined by \begin{align}
     r &:= 48(C+C_2)/\sigmamin^2 & R &:= r\cdot p^6. 
 \end{align} The remaining proof bounds the three integrals separately.

 \paragraph{The integral \protect{$I_3$}} Let $t\in (R^k, \infty)$, meaning we must bound the probability that $S(\mathbf{v})$ is close to zero. For arbitrary $m\in\mathbb{R}$ and $\tau\geq0$, we have that\begin{equation}\label{eq:proof_unifIntegr1}
     \pr{(Y_i - m)^2\leq \tau} = \pr{|Y_i - m|\leq \sqrt{\tau}} \leq 2 c \sqrt{\tau}.
 \end{equation} We aggregate this bound on an independent subset of $(Y_{i})_{i=1,\ldots, p}$ of size $n:= \lfloor \tfrac{p-1}{M+1}\rfloor + 1$. Note $n\geq 4k$ and $(n-1)(M+1) + 1 \leq p$ by assumption on $p$. By $M$-dependence, $(Y_{j(M+1) + 1})_{j=0,\ldots, n-1}$ are independent, and thus \begin{multline}\label{eq:proof_unifIntegr2}
     \pr{S(\mathbf{v}) \leq s\cdot t ^{-1/k}} \leq
     \pr{\bigcap_{j=0,\ldots, n-1} \{(Y_{j(M+1) + 1} - v_{j(M+1) + 1})^2 \leq s\cdot t ^{-1/k}\}} \\
     =\prod_{j=0}^n \pr{(Y_{j(M+1) + 1} - v_{j(M+1) + 1})^2 \leq s\cdot t ^{-1/k}} \\
     \overset{\eqref{eq:proof_unifIntegr1}}{\leq}
     \left(2c\sqrt{s\cdot t ^{-1/k}}\right)^n = (2c)^n s^{n/2}t^{-\tfrac{n}{2k}}.
 \end{multline} Importantly, $t^{-\tfrac{n}{2k}}$ is integrable on $(R^k,\infty)$: \begin{multline*}
     \int_{R^k}^\infty t^{-\frac{n}{2k}}\dd{t} = \frac{1}{-\frac{n}{2k} + 1}\left[t^{-\frac{n}{2k} + 1}\right]_{R^k}^\infty \overset{\frac{-n}{2k} + 1 < 0}{=} 
     \frac{1}{\frac{n}{2k} - 1} R^{-\frac{n}{2} + k} \\
     = \frac{1}{\frac{n}{2k} - 1} p^{-3n + 6k} r^{-\frac{n}{2} + k} \overset{n \geq 4k}{\leq} p^{-3n + 6k} r^{-\frac{n}{2} + k}.
 \end{multline*} 
 Thus, we can bound $I_3$ as follows: 
 \begin{multline}\label{eq:proof_unifIntegr3}
     I_3 \overset{\eqref{eq:proof_unifIntegr2}}{\leq} \int_{R^k}^\infty (2c)^n s^{n/2} t^{-\frac{n}{2k}}\dd{t} \leq
     (2c)^n s^{n/2} p^{-3n + 6k} r^{-\frac{n}{2} + k}\\
     \overset{\eqref{eq:proof_unifIntegr0}}{\leq}
     (2c)^n
     (2 (C + C_2))^{n/2} p^{-2.5n + 6k} r^{-\frac{n}{2} + k} =
     r^k\left(\frac{2c\sqrt{2 (C + C_2)}}{p\sqrt{r}}\right)^n p^{-3n/2 + 6k}.
 \end{multline}

 \paragraph{The integral \protect{$I_2$}} Let $t\in (r^k, R^k)$. Then, \begin{equation}\label{eq:proof_unifIntegr7}
    \pr{S(\mathbf{v}) \leq s\cdot t^{-1/k}} \leq \pr{S(\mathbf{v}) \leq \frac{s}{r}}.
\end{equation} We control the right hand side probability by a Hoeffding-type bound allowing for partial dependence among the summands. Since Hoeffding's inequality requires bounded random variables, we first approximate $S(\mathbf{v})$ and $s$ by quantities involving bounded random variables. 

According to Lemma \ref{lemmaBoundedVarsVarianceApprox}, for $B:= \max(\sqrt{2C_2}, \bar{B})$ with some $\bar{B}$ only depending on global constants, it holds that \begin{equation*}
        \var{\sign(Y_i)\min(|Y_i|, B)} \geq \frac{\sigmamin^2}{2}
    \end{equation*} for all $1\leq i\leq p$.
We set $\mathcal{I}:= \{1\leq i\leq p \mid |v_{i}|\leq B\}$. Note that $\mathcal{I}$ contains at least $\lfloor\tfrac{p}{2}\rfloor$ elements. Otherwise, \begin{equation*}
    \|\mathbf{v}\|^2 \geq \sum_{i\in\{1,\ldots,p\}\setminus\mathcal{I}} v_{i}^2 > \sum_{i\in\{1,\ldots,p\}\setminus\mathcal{I}} (\sqrt{2C_2})^2 \geq \bigg\lceil\frac{p}{2}\bigg\rceil 2C_2 \geq C_2 p,
\end{equation*} a contradiction to the theorem assumptions.
Defining \begin{align*}
     A_{i} &:=( \sign(Y_i)\min(|Y_i|, B) - v_{i})^2, &a &:= \ex{\sum_{i\in\mathcal{I}} A_{i}},
 \end{align*} we conclude by the previous discussion that\begin{equation}\label{eq:proof_unifIntegr8}
     a \geq \sum_{i\in\mathcal{I}}\var{\sign(Y_i)\min(|Y_i|, B)} \geq \bigg\lfloor\frac{p}{2}\bigg\rfloor \cdot \frac{\sigmamin^2}{2}.
 \end{equation}
 Note that for $i\in\mathcal{I}$, we have $A_{i} \leq (Y_i - v_{i})^2$. Consequently, $S(\mathbf{v}) \geq \sum_{i\in\mathcal{I}} (Y_i - v_{i})^2 \geq \sum_{i\in\mathcal{I}} A_{i}$ and we have that\begin{equation}\label{eq:proof_unifIntegr9}
     \pr{S(\mathbf{v}) \leq \frac{s}{r}} \leq \pr{ \sum_{i\in\mathcal{I}} A_{i} \leq \frac{s}{r}} = \pr{\sum_{i\in\mathcal{I}} A_{i} \leq a - (a -  \frac{s}{r})}.
 \end{equation} Using $r = 48(C+C_2)/\sigmamin^2$, we find: \begin{multline*}
     a -  \frac{s}{r} \overset{\eqref{eq:proof_unifIntegr0}}{\geq} a - \frac{2(C+C_2)p}{r} = a - \frac{1}{24}\sigmamin^2 p \\
     \geq a -\frac{1}{2}\bigg\lfloor\frac{p}{2}\bigg\rfloor  \frac{\sigmamin^2}{2} \overset{\eqref{eq:proof_unifIntegr8}}{\geq} \bigg\lfloor\frac{p}{2}\bigg\rfloor \frac{\sigmamin^2}{4}.
 \end{multline*} Inserting this in \eqref{eq:proof_unifIntegr9}, we obtain that\begin{equation}\label{eq:proof_unifIntegr11}
     \pr{S(\mathbf{v}) \leq \frac{s}{r}} \leq \pr{\sum_{i\in\mathcal{I}} A_{i} \leq a - \bigg\lfloor\frac{p}{2}\bigg\rfloor \frac{\sigmamin^2}{4}}.
 \end{equation}
 We now apply the Hoeffding-type bound. First, note that for $i\in\mathcal{I}$, we have $0\leq A_{i}\leq (2B)^2$ as both $|\sign(Y_i)\min(|Y_i|, B)|\leq B$ and $|v_{i}|\leq B$. 

 The Hoeffding-type bound from \citep[Thm. 2.1]{Janson2004PartlyDependentHoeffding} builds on $\chi(\mathcal{I})$, the number of sets containing fully independent variables needed to cover $\mathcal{I}$. By $M$-dependence, a valid cover of $\{1,\ldots,p\}$ are the $M+1$ index sets $\mathcal{I}_j := \{j+l(M+1) : l=0,\ldots, \big\lfloor \frac{p-j}{M+1}\big\rfloor\}$ for $j=1,\ldots, M+1$. Restricting this covering to $\mathcal{I}$, we obtain $\chi^*(\mathcal{I}) \leq \chi(\mathcal{I})\leq M+1$ in the notation of \citep{Janson2004PartlyDependentHoeffding}. Thus:  
 \begin{multline*}
     \pr{\sum_{i\in\mathcal{I}} A_{i} \leq a - \bigg\lfloor\frac{p}{2}\bigg\rfloor \frac{\sigmamin^2}{4}} \leq \exp\left(-2\frac{(\lfloor\tfrac{p}{2}\rfloor \frac{\sigmamin^2}{4})^2}{(M+1)\sum_{i\in\mathcal{I}} (2B)^4}\right) \\
     \leq\exp\left(-\frac{\lfloor\tfrac{p}{2}\rfloor^2\cdot \sigmamin^4}{8(M+1)p (2B)^4}\right).
 \end{multline*} Inserting this into equation \eqref{eq:proof_unifIntegr11}, we can bound $I_2$ as follows: \begin{multline}\label{eq:proof_unifIntegr12}
     I_2 \overset{\eqref{eq:proof_unifIntegr7}}{\leq} \int_{r^k}^{R^k}\pr{S(\mathbf{v}) \leq \frac{s}{r}}\dd{t} \leq
     \pr{S(\mathbf{v}) \leq \frac{s}{r}} \cdot R^k \\
     \leq
     \exp\left(-\frac{\lfloor\tfrac{p}{2}\rfloor^2\cdot \sigmamin^4}{8(M+1)p (2B)^4}\right) \cdot p^{6k} \cdot r^k.
 \end{multline}
 
 \paragraph{The integral \protect{$I_1$}} Bounding the probability term by one, we obtain that \begin{equation}\label{eq:proof_unifIntegr13}
     I_1 = \int_{0}^{r^k} \pr{S(\mathbf{v}_p) \leq s_p t^{-1/k}}\dd{t} \leq r^k.
 \end{equation}

 \paragraph{Final step} Inserting \eqref{eq:proof_unifIntegr3}, \eqref{eq:proof_unifIntegr12}, \eqref{eq:proof_unifIntegr13} into \eqref{eq:proof_unifIntegr_decomp}, we have shown: \begin{multline*}
     \ex{\left(\frac{s}{S(\mathbf{v})}\right)^k} \leq r^k\left(\frac{2c\sqrt{2 (C + C_2)}}{p\sqrt{r}}\right)^n p^{-3n/2 + 6k} \\
     + r^k\cdot\exp\left(-\frac{\lfloor\tfrac{p}{2}\rfloor^2\cdot \sigmamin^4}{8(M+1)p (2B)^4}\right) \cdot p^{6k} + r^k.
 \end{multline*} Studying the right hand side for large $p$, we note that the fraction in the first term becomes smaller than $1$ and the exponent is $n\geq 4k > 0$; further by $n\geq 4k$ we have $-3n/2 + 6k \leq 0$, and finally the exponential decay by $\lfloor\tfrac{p}{2}\rfloor^2/p$ is faster than the growth of $p^{6k}$. Consequently, the supremum of the right hand side over $p\in\mathbb{N}$ is finite. Further, as the right hand side only depends on $c, C, C_2, k,\sigmamin^2,M$, and $r, n, B$ are functions of $M, q, p, C, C_2, \sigmamin^2$ only, this finite supremum only depends on those quantities.
\end{proof}

\subsection{Proof of Lemma \ref{lemmaGmedComponentsBounded}}
\lemmaGmedComponentsBounded*
\begin{proof}
    First, consider $p\geq 4(M+1)+1$ and $1\leq i\leq p$. Taking the result of Lemma \ref{lemma_basic_inequality} component-wise, we find: \begin{equation}\label{eq:proof_GmedCompBound1}
        |\mu_i - m_{i}| \leq \ex{\frac{1}{\|\bfX - \mathbf{m}\|}}^{-1} \left|\ex{\frac{X_i - \mu_i}{\|\bfX - \mathbf{m}\|}}\right|.
    \end{equation}
    Applying Jensen's inequality with the convex function $1/\sqrt{x}$, we have that\begin{equation*}
        \frac{1}{\sqrt{\ex{\|\bfX - \mathbf{m}\|^2}}}\ex{\frac{1}{\|\bfX - \mathbf{m}\|}}^{-1} \leq \frac{\sqrt{\ex{\|\bfX-\mathbf{m}\|^2}}}{\sqrt{\ex{\|\bfX - \mathbf{m}\|^2}}} = 1.
    \end{equation*}
    Further, by the Cauchy-Schwarz inequality we find:  \begin{multline*}
        \sqrt{\ex{\|\bfX - \mathbf{m}\|^2}}\left|\ex{\frac{X_i - \mu_i}{\|\bfX - \mathbf{m}\|}}\right| \leq \sqrt{\ex{(X_i -\mu_i)^2}} \cdot\\
        \sqrt{\ex{\frac{\ex{\|\bfX - \mathbf{m}\|^2}}{\|\bfX - \mathbf{m}\|^2}}} \overset{\text{Lemma }\ref{lemmaUnifIntegr}}{\leq} \sqrt{C}\cdot \sqrt{\beta(1, 2(C^2 + C), \mathcal{D})}
    \end{multline*} where we applied Lemma \ref{lemmaUnifIntegr} with $k=1$, using that $\|\mathbf{m}\|^2\leq 2 (C^2 + C) p$ by Lemma \ref{lemmaInitialGmedMeanBound} (see the proof of Lemma \ref{lemma_basic_inequality}). Inserting both the previous bounds into equation \eqref{eq:proof_GmedCompBound1} yields the claim for $p\geq 4(M+1)+1$.

    When $p< 4(M+1)+1$, use Lemma \ref{lemmaInitialGmedMeanBound} and properties of $\mathcal{D}$ to bound \begin{equation*}
        |\mu_i - m_{i}| \leq \|\bs{\mu} - \mathbf{m}\| \leq \sqrt{Cp} \leq \sqrt{C\cdot(4(M+1)+1)},  
    \end{equation*} which finalizes the proof.
\end{proof}

\subsection{Additional Lemmas}
\begin{lemma}\label{lemma_exists_unique_gmed}
    Let the $p$-variate $\bfX$ belong to $\mathcal{D}(M, q, c, \sigmamin, C)$. Then, when $p\geq M+2$, there exists a unique geometric median of $\bfX$.
\end{lemma}
\begin{proof}
    For $\mathbf{m}, \mathbf{m}^\prime\in\mathbb{R}^p$, we have that \begin{equation*}
        |\ex{\|\mathbf{X} - \mathbf{m}\|} - \ex{\|\bfX - \mathbf{m}^\prime\|}|\leq \ex{\|\mathbf{m}-\mathbf{m}^\prime\|} = \|\mathbf{m}-\mathbf{m}^\prime\|
    \end{equation*}
    by the inverse triangle inequality, hence $\mathbf{m}\mapsto\ex{\|\bfX - \mathbf{m}\|}$ is continuous. Also, $\ex{\|\bfX - \mathbf{m}\|}$ diverges to $\infty$ as $\|\mathbf{m}\|\to\infty$, since \begin{equation*}
        \ex{\|\bfX - \mathbf{m}\|} \geq \ex{|\|\bfX\| - \|\mathbf{m}\||}\geq \|\mathbf{m}\| - \ex{\|\bfX\|}.
    \end{equation*}
    Since continuous functions admit minimizers on compact sets, a geometric median exists.

    When $p\geq M+2$, then $X_1$ and $X_{M+2}$ are independent, which ensures the distribution of $\bfX$ does not concentrate on a one-dimensional linear subspace. By \citep{Milasevic1987}, the geometric median is unique. 
\end{proof}

\begin{lemma}\label{lemma_basic_inequality}
    Let the $p$-variate $\bfX$ belong to $\mathcal{D}(M, q, c, \sigmamin, C)$ with $q>2$. Let $\bs{\mu}:=\ex{\bfX}$ and let $\mathbf{m}$ be the geometric median of $\bfX$. Then, when $p\geq 2(M+1)+1$,
    \begin{equation*}
        \mathbf{m} - \bs{\mu} = \ex{\frac{1}{\|\bfX - \mathbf{m}\|}}^{-1}\ex{\frac{\bfX - \bs{\mu}}{\|\bfX - \mathbf{m}\|}}
    \end{equation*}
\end{lemma}
\begin{proof}
    Let $G(\mathbf{m}):= \|\bfX - \mathbf{m}\|$ for $\mathbf{m}\in\mathbb{R}^p$. We have for almost all $\mathbf{m}$ that \begin{equation*}
        |\partial_i G(\mathbf{m})| = \left|\frac{(m_i - X_i)}{\|\bfX - \mathbf{m}\|}\right| \leq \frac{\|\bfX - \mathbf{m}\|}{\|\bfX - \mathbf{m}\|} = 1.
    \end{equation*}
    Consequently, by
    \citep[Ch. 3.3, Ex. 3.3]{Shorack2017}, $\ex{G(\mathbf{m})}$ is differentiable and \begin{equation*}
        \nabla \ex{G(\mathbf{m})} = \ex{\frac{\mathbf{m} - \bfX}{\|\bfX - \mathbf{m}\|}}.
    \end{equation*} The geometric median $\mathbf{m}$ of $\bfX$ (see Lemma \ref{lemma_exists_unique_gmed}) optimizes $\mathbf{m}\mapsto\ex{G(\mathbf{m})}$, implying that\begin{equation}\label{eq:proof_basic_inequality}
        0 = \ex{\frac{\bfX - \mathbf{m}}{\|\bfX - \mathbf{m}\|}}.
    \end{equation}
    By Lemma \ref{lemmaUnifIntegr} with $k=\tfrac{1}{2}$ (note $\|\mathbf{m}\|^2\leq 2 (\|\bs{\mu}\|^2 + \|\mathbf{m} -\bs{\mu}\|^2) \leq 2 (C^2 + C)p$ by Lemma \ref{lemmaInitialGmedMeanBound}; both Lemmas are proven separately later), we have for $p\geq 2(M+1)+1$ that \begin{equation*}
        \ex{\frac{1}{\|\bfX - \mathbf{m}\|}} \leq \frac{\beta(0.5, 2 (C^2 + C), \mathcal{D})}{\sqrt{\ex{\|\bfX - \mathbf{m}\|^2}}} \leq \frac{\beta(0.5, 2 (C^2 + C), \mathcal{D})}{\sigmamin \sqrt{p}} < \infty
    \end{equation*} by $\ex{\|\bfX - \mathbf{m}\|^2}\geq \ex{\|\bfX - \bs{\mu}\|^2}$ due to MSE optimality of the mean. Therefore, adding $0=\ex{\tfrac{\bs{\mu} - \bs{\mu}}{\|\bfX - \mathbf{m}\|}}$ to \eqref{eq:proof_basic_inequality}, we find: \begin{equation*}
        \ex{\frac{\mathbf{m} - \bs{\mu}}{\|\bfX - \mathbf{m}\|}} = \ex{\frac{\bfX - \bs{\mu}}{\|\bfX - \mathbf{m}\|}}.
    \end{equation*} Pulling constants out of the expectations and dividing by $\ex{\tfrac{1}{\|\bfX - \mathbf{m}\|}} >0$ (otherwise $\bfX=\mathbf{m}$ almost surely, a contradiction to $\sigmamin>0$), we have shown the claim.
\end{proof}

\begin{restatable}{lemma}{lemmaInitialGmedMeanBound}\label{lemmaInitialGmedMeanBound}
    Let the $p$-variate $\bfX$ belong to $\mathcal{D}(M, q, c, \sigmamin, C)$. Then, $\|\mathbf{m} - \bs{\mu}\|\leq \sqrt{Cp}$.
\end{restatable} 
\begin{proof}
    By assumption, we have $\var{X_i}\leq C$ for all $i\in\mathbb{N}$. Using the triangle inequality, optimality of the geometric median, Jensen's inequality and the variance bound yields that \begin{multline*}
        \|\mathbf{m} - \bs{\mu}\| = \|\ex{\mathbf{m} - \bs{\mu}}\| = \|\ex{\mathbf{m} - \bfX}\| \leq \ex{\|\mathbf{m} - \bfX\|} \\\leq
        \ex{\|\bs{\mu} - \bfX\|} \leq \sqrt{\ex{\|\bs{\mu} - \bfX\|^2}} \leq \sqrt{Cp}
    \end{multline*} as desired.
\end{proof}

\begin{restatable}{lemma}{lemmaBasicEqImproved}\label{lemmaBasicEqImproved}
    Let the $p$-variate $\bfX$ belong to the class $\mathcal{D}(M, q, c, \sigmamin, C)$ with $q>2$. 
    Fix $i\leq p$ and define $\check{\bfX}:= (X_j)_{\{1\leq j\leq p\, :\, |j-i|>M\}}$ and similarly $\check{\mathbf{m}}=(m_j)_{\{1\leq j\leq p\, :\, |j-i|>M\}}$. When $p\geq 4(M+1)$, it holds that \begin{equation*}
        m_{i} - \mu_i = \ex{\frac{1}{\|\bfX - \mathbf{m}\|}}^{-1} \ex{\frac{-(X_i - \mu_i)\sum_{|j - i|\leq M}(X_j - m_{j})^2}{\|\bfX - \mathbf{m}\|\|\check{\bfX} - \check{\mathbf{m}}\|(\|\bfX - \mathbf{m}\| + \|\check{\bfX} - \check{\mathbf{m}}\|)}}.
    \end{equation*}
\end{restatable}
\begin{proof}
    Recall that by Lemma \ref{lemma_basic_inequality} we have: \begin{equation}\label{eq:proof_improvedbasic1}
        m_i - \mu_i = \ex{\frac{1}{\|\bfX - \mathbf{m}\|}}^{-1}\ex{\frac{X_i - \mu_i}{\|\bfX - \mathbf{m}\|}}.
    \end{equation}

    Further, $\check{\bfX}$ and $\check{\mathbf{m}}$ have at least $2(M+1)+1$ components by assumption on $p$. Using Lemma \ref{lemmaGmedComponentsBounded} to control $\|\check{\mathbf{m}}\|^2$, we may apply Lemma \ref{lemmaUnifIntegr} with $k=\tfrac{1}{2}$ to obtain $\ex{\tfrac{1}{\|\check{\bfX} - \check{\mathbf{m}}\|}} < \infty$. Thus: 
    \begin{equation*}
       \ex{\frac{X_i - \mu_i}{\|\check{\bfX} - \check{\mathbf{m}}\|}} = \ex{X_i - \mu_i} \ex{\frac{1}{\|\check{\bfX} - \check{\mathbf{m}}\|}} = 0.
    \end{equation*} For reals $x,y>0$ it holds that \begin{equation*}
        \frac{1}{\sqrt{x}} - \frac{1}{\sqrt{y}} = \frac{\sqrt{y} - \sqrt{x}}{\sqrt{x}\sqrt{y}} = \frac{y - x}{\sqrt{x}\sqrt{y}\,(\sqrt{x} + \sqrt{y})}.
    \end{equation*} Subtracting $0=\ex{\tfrac{1}{\|\bfX - \mathbf{m}\|}}^{-1}\ex{\tfrac{X_i - \mu_i}{\|\check{\bfX} - \check{\mathbf{m}}\|}}$ from the right hand side of \eqref{eq:proof_improvedbasic1} and inserting this algebraic relation yields the claim.
\end{proof}

\begin{restatable}{lemma}{lemmaBoundedVarsVarianceApprox}\label{lemmaBoundedVarsVarianceApprox}
    Let $(Y_i)_{i\in\mathbb{N}}$ be a sequence of real random variables such that $\sup_{i\in\mathbb{N}}\ex{|Y_i|^q}\leq C_1 < \infty$ for some $q>2$ and $\inf_{i\in\mathbb{N}}\var{Y_i}\geq \sigmamin^2$. Then there exists $\bar{B}>0$ depending only $ q, C_1$ and $\sigmamin^2$ such that  \begin{equation*}
        \var{\sign(Y_i)\cdot\min(|Y_i|, B)} \geq \frac{\sigmamin^2}{2}
    \end{equation*} for all $B\geq\bar{B}$ and $i\in\mathbb{N}$.
\end{restatable}
\begin{proof}
    First, recall that for any nonnegative random variable $Y$ and $j>0$, we have $\ex{Y^j}= \int_0^\infty j y^{j-1}\pr{Y > y}\dd{y}$ \citep[Lemma 2.2.13]{DurrettProbabilityBook}. 

    Let $B>0$ be arbitrary. Then, for $j\in \{1,2\}$, $i\in\mathbb{N}$, using Markov's inequality, we have:\begin{multline*}
        \ex{|\sign(Y_i)\cdot\min(|Y_i|, B)|^j\cdot 1_{\{|Y_i| > B\}}} = B^j \pr{|Y_i| > B} =\\
        B^j \pr{|Y_i|^q > B^q} \leq B^j \frac{C_1}{B^q} = C_1 B^{j-q}.
    \end{multline*} Further, note $\sign(Y_i)\cdot\min(|Y_i|, B)\cdot 1_{\{|Y_i| \leq B\}} = Y_i\cdot 1_{\{|Y_i| \leq B\}}$. Thus, we have: \begin{multline*}
        \left|\var{Y_i} - \var{\sign(Y_i)\cdot\min(|Y_i|, B)}\right| = |\ex{Y_i^2} - \ex{Y_i}^2 - \\
        \ex{(\sign(Y_i)\cdot\min(|Y_i|, B))^2} + \ex{\sign(Y_i)\cdot\min(|Y_i|, B)}^2 | \leq \\
        \ex{(\sign(Y_i)\cdot\min(|Y_i|, B))^2 \cdot 1_{\{|Y_i| > B\}}} + \\
        \ex{|\sign(Y_i)\cdot\min(|Y_i|, B)|\cdot 1_{\{|Y_i| > B\}}}\cdot\ex{|Y_i + \sign(Y_i)\cdot\min(|Y_i|, B)|} \leq \\
        C_1 B^{2-q} + C_1 B^{1-q} \cdot (C_1^{1/q} + B) = 2C_1 B^{2-q} + C_1^{1+1/q}B^{1-q}.
    \end{multline*} Letting $B\to\infty$, we obtain that the difference of the variances is less than $\tfrac{\sigmamin^2}{2}$ for $B$ large enough, which, together with the fact that the bound in the preceeding display is monotonically decreasing in $B$, yields the claim.
\end{proof}

\begin{restatable}{lemma}{lemmaNumeratorBound}\label{lemmaNumeratorBound}
    Let the $p$-variate $\bfX$ belong to $\mathcal{D}(M, q, c, \sigmamin, C)$ with $q \geq 3$. Then, for all $3\leq k \leq q$, we have that \begin{equation*}
        \ex{|\,(X_i - \mu_i)\sum_{|j-i|\leq M} (X_j - m_{j})^2\,|^{k/3}} \leq B_4(k, \mathcal{D})
    \end{equation*} for all $1\leq i\leq p$ and a constant $B_4$ only depending on $k$ and the $\mathcal{D}$-parameters.
\end{restatable}
\begin{proof}
    Let $N:= |\{j\in\{1,\ldots,p\} : |j - i|\leq M\}|$ be the number of indices $j$ in the sum. Then, by convexity of $x\mapsto x^{k/3}$, we have that \begin{equation*}
        \ex{|\,(X_i - \mu_i)\sum_{|j-i|\leq M} (X_j - m_{j})^2\,|^{k/3}} \leq N^{\tfrac{k}{3} - 1} \sum_{|j-i|\leq M}\ex{|\,(X_i - \mu_i)(X_j - m_{j})^2\,|^{k/3}}.
    \end{equation*} Further, by Hölder's inequality \begin{multline*}
        \ex{|\,(X_i - \mu_i)(X_j - m_{j})^2\,|^{k/3}} \leq \left(\|X_i-\mu_i\|_{L^k} \cdot \|(X_j - m_j)^2\|_{L^{k/2}}\right)^{k/3}\\
        = \ex{|X_i - \mu_i|^k}^{\tfrac{1}{3}} \cdot \ex{|X_j - m_j|^k}^{\tfrac{2}{3}}.
    \end{multline*} 
    
    Note $|\mu_i|\leq C$ and $|m_{j}|\leq B_1 + C$ according to Lemma \ref{lemmaGmedComponentsBounded}. Further, for any $A\in\mathbb{R}$ we have that\begin{equation*}
        \ex{|X_i - A|^k} \leq \ex{(|X_i| + |A|)^k} \leq 2^{k-1} (\ex{|X_i|^k} + A^k) \leq 2^{k-1}(C+A^k).
    \end{equation*}Taking everything together, we find: \begin{equation*}
        \ex{|\,(X_i - \mu_i)\sum_{|j-i|\leq M} (X_j - m_{j})^2\,|^{k/3}} \leq N^{\tfrac{k}{3} - 1} \cdot N \cdot 2^{k-1}(C+C^k)^{\tfrac{1}{3}}(B_1+2C)^{\tfrac{2}{3}}.
    \end{equation*} Noting $N\leq 2M+1$ yields the claim.
\end{proof}

\begin{lemma}\label{lemma_conv_prob}
    Let $(X_i)_i\in \mathcal{D}(M, q, c, \sigmamin, C)$ with $q> 2$. Then, $\|\bfX - \mathbf{m}_p\|^2/\ex{\|\bfX - \mathbf{m}_p\|^2}\to 1$ in probability.
\end{lemma}
\begin{proof}
    If $q\geq 4$, the claim can be shown using Chebychev's inequality. To allow $q>3$ in Theorem \ref{theoremComponentExpansion}, we instead employ a truncation approach similar to \citep[Sect. 2.2.3]{DurrettProbabilityBook}. For $1\leq i\leq p$, define $A_{p,i}:= (X_i - m_{p,i})^2 \cdot 1_{\{(X_i - m_{p,i})^2 \leq p\}}$ and set $a_p := \ex{\sum_{i=1}^p A_{p,i}}$.

    First, note by the triangle inequality, convexity of $x\mapsto x^q$, and Lemma \ref{lemmaGmedComponentsBounded} that \begin{equation}\label{eq:proof_conv_prob0}
        \ex{|X_i - m_{p,i}|^q} \leq 2^q(\tfrac{1}{2}\ex{|X_i|^q} + \tfrac{1}{2}|m_{p,i}|^q) \leq 2^{q-1} (C + B_1^q) := \tilde{C}
    \end{equation} for $1\leq i \leq p$. Further, for $t>0$ by Markov's inequality,  \begin{multline}\label{eq:proof_conv_prob1}
        \pr{(X_i - m_{p,i})^2 > t} = \pr{|X_i - m_{p,i}|^q > t^{q/2}} \leq \frac{\ex{|X_i - m_{p,i}|^q}}{t^{q/2}} \leq \frac{\tilde{C}}{t^{q/2}}.
    \end{multline} Third, using $q>2$, we make use of the bound \begin{multline}\label{eq:proof_conv_prob2}
        \ex{(X_i - m_{p,i})^2 \cdot 1_{\{(X_i - m_{p,i})^2 > p\}}} = \int_0^p \pr{(X_i - m_{p,i})^2 > p}\dd{\tau} + \\
        \int_p^\infty \pr{(X_i - m_{p,i})^2 > \tau}\dd{\tau} \overset{\eqref{eq:proof_conv_prob1}}{\leq} p^{-q/2+1} \tilde{C} + \int_p^\infty \frac{\tilde{C}}{\tau^{q/2}} \dd{\tau} \\= 
        p^{-q/2+1} \tilde{C} + \frac{1}{\tfrac{q}{2} - 1} p^{-q/2 + 1}\tilde{C} = p^{-q/2 + 1} \left(\tilde{C} + \frac{\tilde{C}}{q/2-1}\right).
    \end{multline}
    To prove convergence in probability, let $t>0$. Then, \begin{multline}\label{eq:proof_conv_prob3}
        \pr{\left|\frac{\|\bfX - \mathbf{m}_p\|^2}{a_p} - 1\right| > t} = \pr{\left|\|\bfX - \mathbf{m}_p\|^2 - a_p\right| > a_p t} \leq \\
        \pr{\|\bfX - \mathbf{m}_p\|^2 \neq \sum_{i=1}^p A_{p,i}} + \pr{\left| \sum_{i=1}^p A_{p,i} - a_p\right| > a_p t}.
    \end{multline}
    We have by inclusion of events that by $q>2$, \begin{multline}\label{eq:proof_conv_prob4}
        \pr{\|\bfX - \mathbf{m}_p\|^2 \neq \sum_{i=1}^p A_{p,i}} \leq \pr{\bigcup_{i=1}^p \{(X_i - m_{p,i})^2 > p\}} \\\overset{\eqref{eq:proof_conv_prob1}}{\leq}
        \sum_{i=1}^p \frac{\tilde{C}}{p^{q/2}} = \tilde{C}p^{-q/2 + 1} \xrightarrow{p\to\infty} 0.
    \end{multline}We treat the second term with Chebycheff's inequality. First, when $q\geq 4$, $\ex{A_{p,i}^2} \leq \ex{(X_i - m_{p,i})^4}\leq \tilde{C}^{4/q} \leq \tilde{C}$ since $\tilde{C}\geq C\geq 1$ and using Jensen's inequality. If $q < 4$, \begin{multline*}
        \ex{A_{p,i}^2} = \ex{((X_i - m_{p,i})^2)^{2-q/2}|X_i - m_{p,i}|^{q}1_{\{(X_i - m_{p,i})^2 \leq p\}}} \\\leq
        p^{2 - q/2} \ex{|X_i - m_{p,i}|^{q}} \leq p^{2 - q/2} \tilde{C}.
    \end{multline*} Thus, in any case $\var{A_{p,i}} \leq \ex{A_{p,i}^2} \leq \tilde{C}\cdot\max(1, p^{2-q/2})$.

    Let $i\in\mathbb{N}$. By $M$-dependence, $\Cov(A_{p,i}, A_{p,j})=0$ for all but $2M+1$ indices $j$. This fact together with $|\Cov(X,Y)|\leq \sqrt{\var{X}\var{Y}}$ gives: \begin{equation*}
        \left|\sum_{j=1}^p \Cov(A_{p,i}, A_{p,j})\right| \leq (2M+1)\cdot \tilde{C}\cdot\max(1, p^{2-q/2}).
    \end{equation*} Thus: \begin{equation*}
        \ex{\left(\sum_{i=1}^p A_{p,i} - a_p\right)^2} = \var{\sum_{i=1}^p A_{p,i}} \leq p \cdot(2M+1)\cdot \tilde{C}\cdot\max(1, p^{2-q/2})
    \end{equation*} Hence, by Chebycheff's inequality, \begin{equation*}
        \pr{\left| \sum_{i=1}^p A_{p,i} - a_p\right| > a_p t} \leq \frac{p \cdot(2M+1)\cdot \tilde{C}\cdot\max(1, p^{2-q/2})}{t^2 a_p^2}.
    \end{equation*} Note that \begin{multline*}
        a_p = \ex{\|\bfX - \mathbf{m}_p\|^2} - \sum_{i=1}^p \ex{(X_i - m_{p,i})^2 \cdot 1_{\{(X_i - m_{p,i})^2 > p\}}} \\\overset{\eqref{eq:proof_conv_prob2}}{\geq}
        p\,\sigmamin^2 - p^{-q/2 + 2}\left(\tilde{C} + \frac{\tilde{C}}{q/2-1}\right) \overset{\text{p large enough}}{\geq} p \frac{\sigmamin^2}{2}.
    \end{multline*} Thus, for $p$ large enough, \begin{equation}\label{eq:proof_conv_prob5}
        \pr{\left| \sum_{i=1}^p A_{p,i} - a_p\right| > a_p t} \leq \frac{1}{t^2}\cdot \frac{4(2M+1)\tilde{C}}{\sigmamin^4} \cdot\max\left(\frac{1}{p}, p^{1 - q/2}\right)\xrightarrow{p\to\infty} 0.
    \end{equation} Combining \eqref{eq:proof_conv_prob5} with \eqref{eq:proof_conv_prob4} in \eqref{eq:proof_conv_prob3}, we have shown that\begin{equation}\label{eq:proof_conv_prob6}
        \frac{\|\bfX - \mathbf{m}_p\|^2}{a_p} \xrightarrow{P} 1
    \end{equation} in probability as $p\to\infty$.
    Further, note that \begin{equation*}
        \frac{\sum_{i=1}^p \ex{(X_i - m_{p,i})^2 \cdot 1_{\{(X_i - m_{p,i})^2 > p\}}}}{\ex{\|\bfX - \mathbf{m}_p\|^2}} \leq \frac{p^{2 - q/2}}{p\,\sigmamin^2} \left(\tilde{C} + \frac{\tilde{C}}{q/2-1}\right) \xrightarrow{p\to\infty} 0,
    \end{equation*} which implies that\begin{equation*}
        \frac{a_p}{\ex{\|\bfX - \mathbf{m}_p\|^2}} = 1 - \frac{\sum_{i=1}^p \ex{(X_i - m_{p,i})^2 \cdot 1_{\{(X_i - m_{p,i})^2 > p\}}}}{\ex{\|\bfX - \mathbf{m}_p\|^2}} \xrightarrow{p\to\infty} 1.
    \end{equation*} Combining this with \eqref{eq:proof_conv_prob6} yields the claim by Slutsky's theorem.
\end{proof}

\section{Further analysis of the simulation results}
Linear Regression slopes for the simulation curves can be found in Table \ref{tab:slopes}. The Pareto(3.1) results are analyzed in more detail in Figure \ref{fig:Pareto3Explanation}.
\begin{table}[h]
    \centering
    \begin{tabular}{ccccc}
          & Exp[0.75, 1.25] & MA(2) & Pareto(3.1) & Pareto[7,8]  \\\hline
        Slope of $\|\hat{\mathbf{m}}_p - \bs{\mu}\|$ & $-0.49$ & $-0.51$ & $-0.28$ & $-0.45$\\
        Slope of Residual & $-1.45$ & $-1.64$ & $-1.03$ & $-1.67$
    \end{tabular}
    \caption{Linear regression slopes to the simulation curves after $p\geq 100$}
    \label{tab:slopes}
\end{table}

\begin{figure}
    \centering
    \begin{subfigure}[b]{0.45\textwidth}
         \centering
         \includegraphics[width=\textwidth]{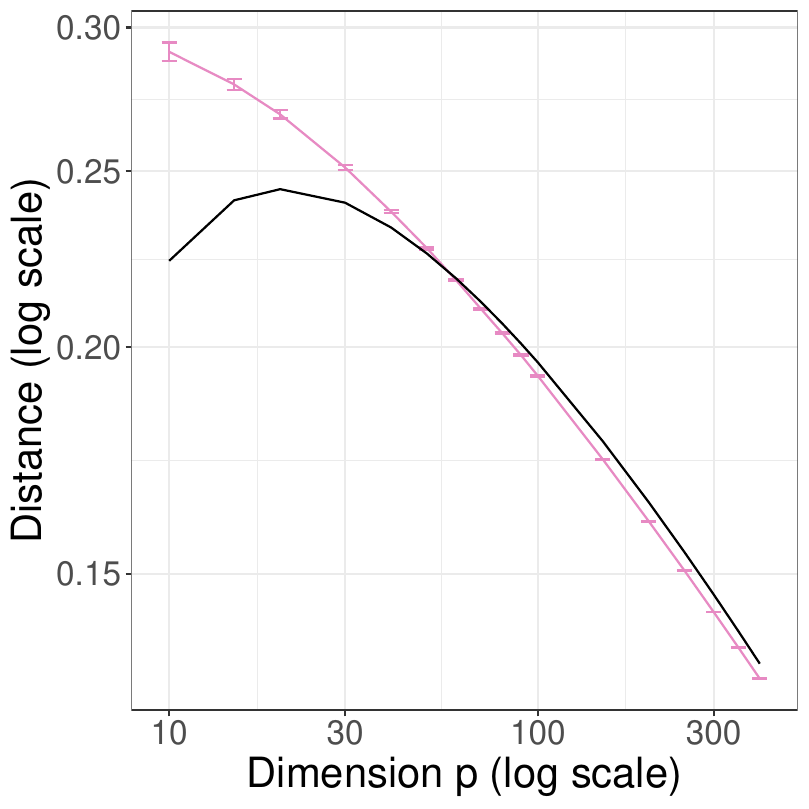}
     \end{subfigure}
     \hfill
     \begin{subfigure}[b]{0.45\textwidth}
         \centering
         \includegraphics[width=\textwidth]{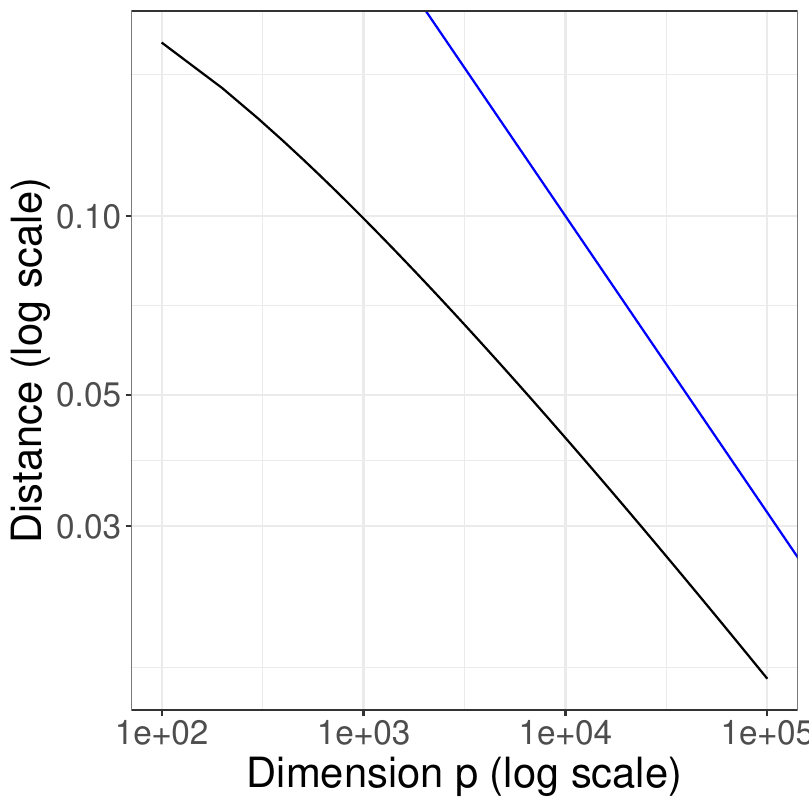}
     \end{subfigure}
    \caption{A closer look at the Pareto(3.1) simulations. Left: $\|\hat{\mathbf{m}}_p - \bs{\mu}\|$ (same as left side of Figure \ref{fig:Simulation}) in pink versus $f(p):=\sqrt{p}\cdot|\tfrac{-19.52}{p} + \tfrac{20.16}{p^{1.03}}|$ in black. Indeed, $f$ explains the decay of $\|\hat{\mathbf{m}}_p - \bs{\mu}\|$ well. Right: Plot $f(p)$ from $100$ to $10^5$ in black and compare with $\mathcal{O}(\tfrac{1}{\sqrt{p}})$ in blue. Even for $p$ so large, $f$ is still not fully parallel to $\mathcal{O}(\tfrac{1}{\sqrt{p}})$.}
    \label{fig:Pareto3Explanation}
\end{figure}

\end{document}